%% file: catO.tex
\documentclass[a4paper, english]{amsart}

\title{Traces, Schubert calculus, and Hochschild cohomology of category \texorpdfstring{$\catO$}{O}}
\author{Clemens Koppensteiner}
\address{Mathematical Institute, University of Oxford, Andrew Wiles Building, Radcliffe Observatory Quarter, Woodstock Road, OX2 6GG, Oxford UK}
\email{clemens.koppensteiner@maths.ox.ac.uk}
\thanks{The author was supported by the EPSRC grant EP/R045038/1.}

\usepackage[utf8]{inputenc}
\usepackage[T1]{fontenc}
\usepackage{lmodern}
\usepackage{microtype}

\usepackage{amsmath,amsfonts,amssymb}
\usepackage{mathtools,extpfeil,xfrac}
\usepackage{amsthm,thmtools}

\usepackage{tikz-cd}

\usepackage{enumitem}
\setlist[enumerate,1]{label=(\roman*)}
\setlist[enumerate,2]{label=(\alph*)}
\setlist[enumerate,3]{label=(\Alph*)}
\setlist[enumerate,4]{label=\arabic*.}

\usepackage[
    backend=biber,
    style=alphabetic,
    maxnames=100,
    maxalphanames=100,
]{biblatex}

\DeclareLabelalphaTemplate{
    \labelelement{
        \field[final]{shorthand}
        \field{label}
        \field[varwidthlist,strside=left,names=4]{labelname}
    }
}

\DeclareFieldFormat{labelalpha}{#1}
\DeclareFieldFormat{extraalpha}{#1}

\usepackage{mparhack}  
\usepackage{todonotes}
\usepackage{csquotes}
\usepackage[overload]{textcase}
\usepackage{url}

\usepackage[colorlinks=false,unicode=true,bookmarksdepth=2,pdfusetitle]{hyperref}
\usepackage{bookmark}
\hypersetup{pdfauthor={Clemens Koppensteiner}}

\theoremstyle{plain}
\newtheorem{Theorem}{Theorem}[section]
\newtheorem*{Theorem*}{Theorem}

\newtheorem{Proposition}[Theorem]{Proposition}
\newtheorem{Corollary}[Theorem]{Corollary}
\newtheorem{Lemma}[Theorem]{Lemma}

\theoremstyle{definition}

\theoremstyle{remark}
\newtheorem{Remark}[Theorem]{Remark}

\newtheorem{Example}[Theorem]{Example}

\input{macros.tex}

\begin{document}

\begin{abstract}
    We discuss how the Hochschild cohomology of a dg category can be computed as the trace of its Serre functor.
    Applying this approach to the principal block of the Bernstein--Gelfand--Gelfand category $\catO$, we obtain its Hochschild cohomology as the compactly supported cohomology of an associated space.
    Equivalently, writing $\catO_0$ as modules over the endomorphism algebra $A$ of a minimal projective generator, this is the Hochschild cohomology of $A$.
    In particular our computation gives the Euler characteristic of the Hochschild cohomology of $\catO_0$ in type A.
\end{abstract}

\maketitle


\section{Introduction}\label{sec:intro}

Let $G$ be a semisimple complex algebraic group $G$ with corresponding Lie algebra $\mathfrak g$.
Of fundamental importance to the representation theory of $G$ is the category of perverse sheaves on the flag variety $X$ of $G$ smooth along the Schubert stratification.
Via Beilinson--Bernstein localization and the Riemann--Hilbert correspondence this category is equivalent to the principal block of $\catO_0$ of the Bernstein--Gelfand--Gelfand category $\catO$.

Given their central importance in representation theory, one would like to compute various fundamental invariants of these categories.
In this paper we present a geometric approach to computing the Hochschild cohomology of $\catO_0$ (or rather of the dg enhanced derived category of $\catO_0$).
As $\catO_0$ is equivalent to the category of finitely generated modules over the (finite dimensional) endomorphism algebra $A$ of a minimal projective generator of $\catO_0$, this is the same as the Hochschild cohomology of $A$.

By \cite[Proposition~1.5]{BeilinsonBezrukavnikovMirkovic:2004:TiltingExercises}, the bounded derived category of the category of perverse sheaves as above coincides with the bounded derived category $\catDbConstr$ of Schubert-constructible sheaves on the flag variety.
In Section~\ref{sec:traces} we describe how the Hochschild cohomology of a dg category can be computed as the trace of its Serre functor.
As the Serre functor of $\catDbConstr$ is described in \cite{BeilinsonBezrukavnikovMirkovic:2004:TiltingExercises}, this approach applies to our categories.
In Theorem~\ref{thm:HH_as_Hc}, we present the Hochschild cohomology as compactly supported cohomology of an auxiliary space $\HochSpace$ associated to $G$.

The space $\HochSpace$ can be stratified by intersections of Schubert cells.
Thus knowledge of the topology of such intersections would allow us to compute the Hochschild cohomology of $\catO_0$.
In particular, for $G = \SL n$ enough topological information is known to compute its Euler characteristic.

\begin{Theorem*}
    For $G = \SL n$, the Euler characteristic of the Hochschild cohomology of $\catDbConstr$ -- or equivalently of the algebra $A$ -- is
    \[
        \chi\bigl(\HH^\bullet(\catDbConstr)\bigr) = |W| = n!,
    \]
    where $W = \symgroup n$ is the Weyl group of $G$.
\end{Theorem*}

\begin{Remark}\label{rem:cohomology_in_HH}
    One notes that the cohomology of the flag variety, having a basis given by classes of Schubert varieties, also has Euler characteristic $|W|$.
    Indeed there are two natural monomorphisms from the cohomology of $X$ to $\HH^\bullet(\catDbConstr)$:
    one the one hand one has 
    \[
        \HH^0(\catDbConstr) \cong \bigoplus_n H^n(X,\, \CC),
    \]
    see for example \cite[Theorem~1]{Stroppel:2009:ParabolicCategoryO}, on the other hand in Corollary~\ref{cor:map_from_cohomology} we show the existence of a monomorphism 
    \[
        H^\bullet(X)^\op \hookrightarrow \HH^\bullet(\catDbConstr).
    \]
\end{Remark}

More generally, we conjecture that in any type and for any parabolic subgroup $P$ of $G$ with corresponding Lie algebra $\mathfrak p$, the Euler characteristic of the Hochschild cohomology of the principal block $\catO_0^{\mathfrak p}$ of the parabolic BBG category $\catO^{\mathfrak p}$ is $|W/W_P|$, where $W_P \subseteq W$ is the subgroup associated to $P$.
We intend to return to this question, as well as the study of singular blocks of $\catO$, in future work.
Let us just note here that in both of these situations at least an algebraic description of the Serre functor of the corresponding derived category is know \cite{Stroppel:2009:ParabolicCategoryO,MazorchukStroppel:2008:ProjectiveInjectiveModulesSerreFunctorsEtc}.

\section{Dualizable categories, Serre functors, and Hochschild cohomology}\label{sec:traces}

In this section $k$ is any field of characteristic $0$, which in the following sections will be specialized to $k = \CC$.
We fix a model for the homotopy category of dg categories over $k$ with continuous (i.e., colimit preserving) functors and tensor product of dg categories as monoidal structure.
The reader may user their preferred realization, but for concreteness following \cite[Section~1.10]{GaitsgoryRozenblyum:2017:StudyInDAG:1} we let $\DGcatcont$ be the monoidal $(\infty,2)$-category of presentable stable $(\infty,1)$-categories enhanced in the derived $(\infty,1)$-category $\catVect$ of vector spaces, with continuous functors.
We write $\otimes = \otimes_{\catVect}$ for the symmetric monoidal product of dg categories with unit $\catVect$.

Any dg category $\cat C\in\DGcatcont$ is in particular cocomplete, i.e., it admits (small) filtered colimits.
An object $c$ of $\cat C$ is called \emph{compact} if the functor $\Hom_{\cat C}(c, -)$ preserves such filtered colimits.
We write $\cat C^c$ for the full subcategory of compact objects.
The category $\cat C$ is \emph{compactly generated} if it does not properly contain a full cocomplete stable subcategory that contains all compact objects.
This is equivalent to $\cat C^c$ be a set of weak generators for $\cat C$ \cite[Proposition~1.5.4.5]{GaitsgoryRozenblyum:2017:StudyInDAG:1}.
If $\cat C$ and $\cat D$ are compactly generated, then $\cat C \otimes \cat D$ is compactly generated by objects of the form $c \boxtimes d$ with $c \in \cat C^c$ and $d \in \cat D^c$ \cite[Proposition~1.10.5.7]{GaitsgoryRozenblyum:2017:StudyInDAG:1}.
Moreover, if $c, c' \in \cat C^c$ and $d, d' \in \cat D^c$, then
\[
    \Hom_{\cat C \otimes \cat D}(c \boxtimes d,\, c' \boxtimes d') = \Hom_{\cat C}(c,\, c') \otimes \Hom_{\cat D}(d,\, d').
\]
A continuous functor $F\colon \cat C \to \cat D$ always has a right adjoint $F^R$.
This right adjoint is itself continuous if and only if $F$ preserves compactness \cite[Lemma~1.7.1.5]{GaitsgoryRozenblyum:2017:StudyInDAG:1}.

A dg category $\cat C$ is called \emph{dualizable}, if it is so as an object of the monoidal category $\DGcatcont$.
Thus there exists a dual category $\cat C^\dual$ and continuous functors
\[     
    \coev_{\cat C} \colon \catVect \to \cat C \otimes \cat C^\dual, \qquad
    \ev_{\cat C} \colon \cat C^\dual \otimes \cat C \to \catVect,
\]
such that the composition
\[
    \cat C
    \xrightarrow{\coev_{\cat C} \otimes \id_{\cat C}}
    \cat C \otimes \cat C^\dual \otimes \cat C
    \xrightarrow{\id_{\cat C} \otimes \ev_{\cat C}}
    \cat C
\]
is isomorphic to $\id_{\cat C}$ and
\[
    \cat C^\dual
    \xrightarrow{\id_{\cat C^\dual} \otimes \coev_{\cat C}}
    \cat C^\dual \otimes \cat C \otimes \cat C^\dual
    \xrightarrow{\ev_{\cat C} \otimes \id_{\cat C^\dual}}
    \cat C^\dual
\]
is isomorphic to $\id_{\cat C^\dual}$.

If $\cat C$ is compactly generated, then $\cat C$ is dualizable and one can identify $(\cat C^\dual)^c$ with $(\cat C^c)^{\mathrm{op}}$ \cite[Proposition~1.7.3.2]{GaitsgoryRozenblyum:2017:StudyInDAG:1}.
Denote this identification by $c \mapsto c^\dual$ for $c \in \cat C^c$.
One has a canonical isomorphism
\[
    \ev_{\cat C}(a^\dual \boxtimes b) \cong \Hom_{\cat C}(a, b)
\]
for $a,\, b \in \cat C^c$.
We note that $(\cat C^\dual)^\dual \cong \cat C$.

Let $\sigma_{\cat C}$ denote both the identification $\cat C \otimes \cat C^\dual \cong \cat C^\dual \otimes \cat C$ and its inverse, given by interchanging the factors.
We will write $\coev_{\cat C}^\sigma$ for the composition $\sigma \circ \coev_{\cat C}$ and $\ev_{\cat C}^\sigma$ for $\ev_{\cat C} \circ \sigma$.
Further we set $u_{\cat C} = \coev_{\cat C}(k)$.

For any dualizable dg category $\cat C$, the category $\cat C \otimes \cat C^\dual$ also dualizable, with dual $(\cat C \otimes \cat C^\dual)^\dual \cong \cat C^\dual \otimes \cat C$ and evaluation map
\[
    \ev_{\cat C \otimes \cat C^\dual} = (\ev_{\cat C} \otimes \ev_{\cat C^\dual}) \circ \sigma_{2,3} \colon \cat C^\dual \otimes \cat C \otimes \cat C \otimes \cat C^\dual \to \catVect,
\]
where $\sigma_{2,3}$ exchanges the second and third factor.

For dualizable dg categories $\cat C$ and $\cat D$, we have an identification $\FunctCont(\cat C,\, \cat D) \cong \cat C^\dual \otimes \cat D$.
Explicitly, to a continuous functor $F\colon \cat C \to \cat D$ we associate the kernel
\[
    K_F = (\id_{\cat C^\dual} \otimes F )( \sigma_{\cat C}(u_{\cat C})) \in \cat C^\dual \otimes \cat D,
\]
and to a kernel $K \in \cat C^\dual \otimes \cat D$ we associate the composition
\[
    \cat C
    \xrightarrow{\id_{\cat C} \otimes K}
    \cat C \otimes \cat C^\dual \otimes \cat D
    \xrightarrow{\ev_{\cat C}^\sigma \otimes \id_{\cat D}}
    \cat D.
\]
In particular, the kernel for the identity endofunctor is $\sigma_{\cat C}(u_{\cat C}) \in \cat C^\dual \otimes \cat C$.
If $\cat C$ and $\cat D$ are dualizable dg categories, the identification
\[
    \FunctCont(\cat C,\, \cat D) \cong \cat C^\dual \otimes \cat D \cong \FunctCont(\cat D^\dual,\, \cat C^\dual)
\]
assigns to each functor $F\colon \cat C \to \cat D$ its dual $F^\dual\colon \cat D^\dual \to \cat C^\dual$.

A dualizable dg category $\cat C$ is called \emph{2-dualizable} if $\coev_{\cat C}$ has both a left and a right adjoint in $\FunctCont(\cat C^\dual \otimes \cat C,\, \catVect)$.
Equivalently one can ask for $\ev_{\cat C}$ to have continuous adjoints.
It follows that $\ev_{\cat C}$ preserves compactness and hence that $\cat C^c$ is proper, i.e.~has bounded Hom-complexes with finite dimensional cohomology spaces.

\begin{Proposition}\label{prop:Serre}
    Let $\cat C$ be a compactly generated dualizable dg category.
    If $\cat C$ is 2-dualizable then there exist an autoequivalence $S$ of $\cat C$ and quasi-isomorphisms
    \[
        \eta_{a,b} \colon \Hom_{\cat C}(a,\, b) \cong \Hom_{\cat C}(b,\, Sa)^*.
    \]
    functorial in $a,\, b \in \cat C^c$.
\end{Proposition}

One calls $S$ the \emph{Serre functor} of $\cat C^c$ \cite{BondalKapranov:1989:RepresentableFunctorsSerreFunctors}.
Up to isomorphism, it is the only autoequivalence of $\cat C^c$ admitting isomorphisms $\eta$ as above \cite[Proposition~1.5]{BondalOrlov:2001:ReconstructionOfAVariety}.

\begin{proof}
    Let us first assume that $\cat C$ is 2-dualizable.
    Then the kernel $u_{\cat C} = \coev_{\cat C}(k)$ in $\cat C^\dual \otimes \cat C$ is compact.
    Thus the endofunctor $\PsId_{\cat C}$ of $\cat C$ given by $u_{\cat C}^\dual$ is an equivalence by \cite[Proposition 5.4.6]{Gaitsgory:kernels}.

    By \cite[Theorem 5.2.3]{Gaitsgory:kernels}, the right adjoint $\coev_{\cat C}^R$ of $\coev_{\cat C}$ has kernel $u_{\cat C}^\dual \in \cat C^\dual \otimes \cat C \cong (\cat C \otimes \cat C^\dual)^\dual \otimes \catVect$.
    Here we use that $\PsId_{\catVect} = \id_{\catVect}$.
    Setting $S = \PsId_{\cat C}^{-1}$ we therefore have
    \[
        \coev_{\cat C}^R \cong \ev_{\cat C}^\sigma \circ (S^{-1} \otimes \id_{\cat C^\dual}).
    \]
    Under the identification $(\cat C \otimes \cat C^\dual)^\dual \cong \cat C^\dual \otimes \cat C$, the functor $\coev_{\cat C}^\dual$ is again given by the kernel $u_{\cat C}^\dual$.
    Therefore we have $\coev_{\cat C}^\dual \cong \coev_{\cat C^\dual}^R$.
    Thus for compact objects $a$ and $b$, \cite[Lemma~1.5.6]{Gaitsgory:kernels} implies that
    \begin{align*}
        \coev_{\cat C}^L(b \boxtimes (Sa)^\dual) & \cong
        \coev_{\cat C^\dual}^R(b^\dual \boxtimes Sa) \\&\cong
        \ev_{\cat C^\dual}^\sigma\circ (\id_{\cat C^\dual} \otimes S^{-1}) (b^\dual \boxtimes Sa) \\&\cong
        \ev_{\cat C^\dual}^\sigma(b^\dual \boxtimes a) \\&\cong
        \Hom_{\cat C}(a,\, b).
    \end{align*}
    By adjunction we therefore have
    \begin{align*}
        \Hom_{\cat C}(a,\, b)^* &\cong
        \Hom_k(\coev_{\cat C}^L(b \boxtimes (Sa)^\dual),\, k) \\& \cong
        \Hom_{\cat C \otimes \cat C^\dual}(b \boxtimes (Sa)^\dual,\, \coev_{\cat C}(k)) \\&\cong
        \ev_{\cat C \otimes \cat C^\dual}(b^\dual \boxtimes Sa \boxtimes \coev_{\cat C}(k)).
    \end{align*}
    Applying the definition of dualizability, this simplifies to
    \[
        \ev_{\cat C}(b^\dual \boxtimes Sa) \cong
        \Hom_{\cat C}(b,\, Sa)
    \]
    as required.
%
\end{proof}

For any dg category $\cat C$, one defines its \emph{Hochschild cohomology} as the endomorphisms of the identity functor:
\[
    \HH^\bullet(\cat C) = \Hom_{\FunctCont}(\id_{\cat C},\, \id_{\cat C}).
\]
One notes that if $\cat C$ is the category of modules over some algebra $A$, this recovers the Hochschild cohomology of $A$.
In particular, $\HH^0(\cat C)$ is the center of $A$.

To compute the Hochschild cohomology, we will connect it to the trace of the Serre functor.
Here, if $F\colon \cat C \to \cat C$ is an endofunctor with kernel $K_F \in \cat C^\dual \otimes \cat C$, one defines the \emph{trace} $\Tr(F)$ as $\ev_{\cat C}(K_F)$, i.e.~as the image of $k$ under the composition
\[
    \catVect
    \xrightarrow{\coev_{\cat C}^\sigma}
    \cat C^\dual \otimes \cat C
    \xrightarrow{\id_{\cat C^\dual} \otimes F}
    \cat C^\dual \otimes \cat C
    \xrightarrow{\ev_{\cat C}}
    \catVect
\]

\begin{Proposition}\label{prop:HH_as_trace}
    For any 2-dualizable dg category $\cat C$ with Serre functor $S$ there exists a canonical quasi-isomorphism
    \[
        \HH^\bullet(\cat C) \cong \Tr(S^{-1}).
    \]
\end{Proposition}

\begin{proof}
    By definition we have
    \[
        \Tr(S^{-1}) \cong
        \Hom_{\catVect}(k,\, \ev_{\cat C} \circ (\id_{\cat C}^\dual \otimes S^{-1}) \circ \coev_{\cat C}^\sigma(k)).
    \]
    As in the proof of Proposition~\ref{prop:Serre}, one shows that the left adjoint of $\ev_{\cat C}$ is given on compact objects by $(\id_{\cat C^\dual} \otimes S^{-1}) \circ \coev_{\cat C^\dual}$.
    Therefore the above becomes
    \[
        \Hom_{\cat C^\dual \otimes \cat C}((\id_{\cat C}^\dual \otimes S^{-1}) \circ \coev_{\cat C^\dual}(k),\, (\id_{\cat C}^\dual \otimes S^{-1}) \circ \coev_{\cat C}^\sigma(k)).
    \]
    Since $S$ and hence also $S^{-1} \otimes \id_{\cat C^\dual}$ are equivalences, this simplifies to
    \begin{align*}
        \Hom_{\cat C^\dual \otimes \cat C}(\coev_{\cat C^\dual}(k),\, \coev_{\cat C}^\sigma(k)) & \cong
        \Hom_{\cat C^\dual \otimes \cat C}(\sigma(u_{\cat C}),\, \sigma(u_{\cat C})) \\&\cong
        \Hom_{\cat C^\dual \otimes \cat C}(u_{\cat C},\, u_{\cat C}) \\&\cong
        \Hom_{\FunctCont}(\id_{\cat C},\, \id_{\cat C}),
    \end{align*}
    which by definition is equal to the Hochschild cohomology of $\cat C$.
\end{proof}

\section{The Hochschild space}

We will be considering various derived categories of sheaves.
When doing so, we will always enhance them to dg categories.
All functors between such categories will be dg enhancements of derived functors, even though we omit the signifiers $\mathbb{L}$ and $\mathbb R$.

Let $G$ be a semisimple complex algebraic group with fixed Borel subgroup $B$, unipotent radical $U$ and Weyl group $W$.
The flag variety $X = G/B$ is stratified by finitely many $U$-orbits $X_w$, called Schubert cells, indexed by $w \in W$.

We write $\catDbDModCh$ for (a dg enhancement of) the bounded derived category of D-modules on $X$ whose cohomology modules have characteristic variety contained in the union of the conormal varieties to Schubert strata.
Clearly such modules are automatically holonomic and hence the Riemann--Hilbert correspondence identifies $\catDbDModCh$ with the (dg enhanced) bounded derived category $\catDbConstr$ of sheaves on $X$ which are constructible with respect to the Schubert stratification.
By Beilinson--Bernstein localization, either of these categories is isomorphic to the derived category of principal block $\catO_0$ of the BGG category $\catO$ associated to $G$.

Further, we write $\catDbStack$ for the bounded derived category of coherent D-modules on the quotient stack $[X/U]$.
Equivalently these are (strongly) $U$-equivariant D-modules on $X$.
Let $\sigma\colon X \to [X/U]$ be the quotient morphism.

For any space $Y$ we write $k_Y$ for the D-module on $Y$ which corresponds to the constant sheaf $\CC_Y$ via the Riemann--Hilbert correspondence.
In the common normalization of the functors this is the structure sheaf of $Y$ placed in cohomological degree $\dim Y$.
For any $Y$ we write $\pi_Y \colon Y \to \pt$ for the structure map.

\begin{Lemma}
    The pullback $\sigma^!\colon \catDbStack \to \catDbDModCoh{X}$ restricts to an equivalence of categories
    \[
        \sigma^!\colon \catDbStack \isoto \catDbDModCh.
    \]
\end{Lemma}

\begin{proof}
    Since $U \cong \as n$ is unipotent, the functor $\pi^!$ is fully faithful.
    Any $U$-equivariant D-module is constant along orbits, and hence contained in $\catDbDModCh$.
    Let $i_w\colon X_w \hookrightarrow X$ be the inclusion.
    Since each $X_w$ is isomorphic to the affine space $\as{l(w)}$, the category $\catDbDModCh$ is generated as a triangulated category by $i_{w\bullet}\sO_{X_w}$ for $w \in W$.
    These modules can be given a trivial $U$-equivariant structure and hence are in the image of $\pi^!$.
    It follows that the essential image of $\pi^!$ is precisely $\catDbDModCh$.
\end{proof}

\begin{Corollary}\label{cor:map_from_cohomology}
    There exists a canonical split monomorphism of dg algebras 
    \[
        H^\bullet(X, \CC)^\op \hookrightarrow \HH^\bullet(\catDbConstr).
    \]
\end{Corollary}

\begin{proof}
    The identifications above give an identification of $\HH^\bullet(\catDbConstr)$ with $\HH^\bullet(\catDbStack)$.
    By \cite[Proposition~4.3]{Koppensteiner:2018:HochschildCohomologyOfTorusEquivariantDModules}, there exists a split monomorphism of dg algebras
    \[
        (p_\bullet k_{[X/U]})^\op \to \HH^\bullet(\catDbStack),
    \]
    where $p\colon [X/U] \to \pt$ is the structure map.
    Let $\pi\colon X \to [X/U]$ be the quotient map.
    As $\pi^\bullet$ is fully faithful, the adjunction morphism $\id_{\catDbStack} \to \pi_\bullet\pi^\bullet$ is an isomorphism.
    Therefore
    \[
        p_\bullet k_{[X/U]} \cong
        p_\bullet\pi_\bullet \pi^\bullet k_{[X/U]} \cong
        (\pi \circ p)_\bullet k_X \cong
        H^\bullet(X, \CC).
        \qedhere
    \]
\end{proof}

Let $X_w^2 = G(X_0, X_w) \cong G(B, wBw^{-1})$ be the $G$-orbit in $X\times X$ corresponding to $w \in  W$.
Thus $X_w^2$ consists of pairs of flags in relative position $w$.
Let $\proj_1$ and $\proj_2$ be the two projections $X_w^2 \to X$.

Define an algebraic variety $\HochSpace$ by
\[
    \HochSpace = 
    \bigl\{
        (a,b,u) \in X \times X \times U :
        (a,b) \in X_{w_0}^2,\, (b,ua) \in X_{w_0}^2
    \bigr\}.
\]

\begin{Theorem}\label{thm:HH_as_Hc}
    The Hochschild cohomology of $\catDbConstr$ is $H^\bullet_c(\HochSpace, \CC)[4l(w_0)]$.
\end{Theorem}

\begin{proof}\leavevmode
    For $w \in W$ and $? \in \{*,!\}$ let $R_w^? = \proj_{2,?}\proj_1^*[l(w)]$ be the intertwining endofunctor of $\catDbConstr$, where $\proj_i\colon X_w^2 \to X$ are the projections and $l(w)$ is the length of the Weyl group element $w$.
    By \cite[Proposition~2.5]{BeilinsonBezrukavnikovMirkovic:2004:TiltingExercises}, the Serre functor of $\catDbConstr$ is isomorphic to $(R_{w_0}^*)^2$. 
    Since $w_0 = w_0^{-1}$, \cite[Fact~2.2b]{BeilinsonBezrukavnikovMirkovic:2004:TiltingExercises} implies that the inverse Serre functor $S^{-1}$ of $\catDbConstr$ is given by $(R_{w_0}^!)^2$.
    
    Set $Z = X_{w_0}^2 \times_X X_{w_0}^2$ and let $s,\,t\colon Z \rightrightarrows X$ be the projections onto the first and last factor:
    
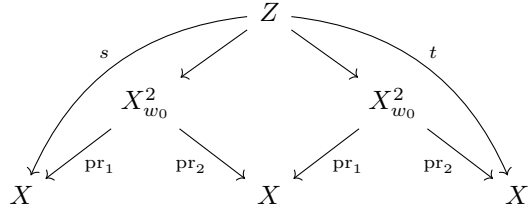
\[
        \begin{tikzcd}
            & & Z \arrow[dl] \arrow[ddll, bend right=30, "s"'] \arrow[dr] \arrow[ddrr, bend left=30, "t"] & & \\
            & X_{w_0}^2 \arrow[dl, "\proj_1"] \arrow[dr, "\proj_2"'] & & X_{w_0}^2 \arrow[dl, "\proj_1"] \arrow[dr, "\proj_2"'] & \\
            X & & X & & X
        \end{tikzcd}
    \]
    Thus by base change the functor $S^{-1} = (R_{w_0}^!)^2$ is given by $t_!s^*[2l(w_0)]$.
    Equivalently, by the Riemann--Hilbert correspondence, the inverse Serre functor of $\catDbDModCh$ is $t_!s^\bullet[2l(w_0)]$

    Adding the equivalence $\sigma^!\colon \catDbStack \to \catDbDModCh$ we see that that in terms of the category $\catDbStack$ the inverse Serre functor is given by
    \[
        S^{-1} =
        \sigma_!t_!s^\bullet\sigma^![2l(w_0)] =
        \sigma_!t_!s^\bullet\sigma^\bullet[4l(w_0)] =
        (\sigma\circ t)_! (\sigma\circ s)^\bullet[4l(w_0)],
    \]
    where one uses that $\sigma\colon X \to [X/U]$ is smooth of relative dimension $\dim U = l(w_0)$.

    The Hochschild cohomology of $\catDbStack$ is the same as that of the corresponding compactly generated cocomplete category $\catDbStackBig$.
    By \cite[Proposition~5.4 and Theorem~3.18]{BenZviNadler:arXiv:CharacterTheoryOfAComplexGroup}, the category $\catDbStackBig$ is 2-dualizable.
    Hence by Proposition~\ref{prop:HH_as_trace}, we can compute its Hochschild cohomology as the trace of the inverse Serre functor.

    By base change (see \cite[Proposition~3.1(2) and Theorem~4.2(2)]{BenZviNadler:arXiv:NonlinearTraces}), the trace of $S^{-1}$ is given by $(\pi_{\res Z\Delta})_!(k_{\res Z\Delta})[4l(w_0)]$, where 
    \[
        \res Z\Delta = Z \times_{[X/U]\times[X/U]} [X/U].
    \]
    It remains to identify the stack $\res Z\Delta$ with $\HochSpace$.
    We consider the commutative diagram with cartesian squares
    \[
        \begin{tikzcd}
            \res Z\Delta \arrow[r] \arrow[d] & X \times U \arrow[r] \arrow[d, "{(\proj_1,a)}"] & {[X/U]} \arrow[d, "\Delta"] \\
            X_{w_0}^2 \times_X X_{w_0}^2 \arrow[r, "{(s,t)}"] & X \times X \arrow[r] & {[X/U] \times [X/U]}
        \end{tikzcd}
    \]
    where $a\colon X \times U \to X$ is the action map.
    Therefore,
    \begin{align*}
        \res Z\Delta &\cong
        \bigl(X_{w_0}^2 \times_X X_{w_0}^2\bigr) \times_{X \times X} \bigl(X \times U\bigr)
        \\&\cong
        \bigl\{
            ((a,b),(c,d), (e,u)) \in X_{w_0}^2 \times X_{w_0}^2 \times X \times U :
            b = c,\, a = e \text{ and } ue= d
        \bigr\}
        \\&\cong
        \bigl\{
            (a,b,u) \in X \times X \times U :
            (a,b) \in X_{w_0}^2,\, (b,ua) \in X_{w_0}^2
        \bigr\} = \HochSpace.
    \end{align*}
    The result thus follows from transporting $(\pi_{\res Z\Delta})_!(k_{\res Z\Delta})[4l(w_0)] \cong \pi_{\HochSpace,!}(k_{\res Z\Delta})[4l(w_0)]$ back to the topological setting via the Riemann--Hilbert correspondence.
\end{proof}

\section{Computations in type A}

According to Theorem~\ref{thm:HH_as_Hc}, in order to compute the Hochschild cohomology of the principal block $\catO_0$, we can equivalently compute the compactly supported cohomology of $\HochSpace$.
Our approach will be to stratify $\HochSpace$ by intersections of Schubert cells.
In type A enough is known about the topology of the occurring intersections to deduce some information on the Hochschild cohomology.

From now on, we set $G = \SL{n}$.
Thus $X$ is the space of full flags in $\CC^n$ and $W \cong \symgroup n$.

\begin{Theorem}\label{thm:HH_Euler}
    For $G = \SL n$, the Euler characteristic of the Hochschild cohomology $\HH^\bullet(\catDbConstr)$ is $n!$.
\end{Theorem}

For two flags $x_1, x_2 \in X$ and a Weyl group element $w \in W$ we denote by $C_w(x_1,x_2)$ the space of all flags $y \in X$ such that $y$ is in relative position $w$ to both $x_1$ and $x_2$.

\begin{Lemma}\label{lem:chi_of_intersection}
    The Euler characteristic of the compactly supported cohomology of $C_{w_0}(x_1,x_2)$ is $0$, except for the case $x_1 = x_2$, in which case it is $1$.
\end{Lemma}

\begin{proof}
    This follows from the computations in \cite{ShapiroVainshtein:1990:EulerCharacteristicsForLinksOfSchubertCells}: with the notation used there we have $C_{w_0}(x_1,x_2) = V_{\overline \sigma}$, where $\overline \sigma \in W$ is the relative position of $x_1$ and $x_2$.
    Then \cite[Lemma~2.3]{ShapiroVainshtein:1990:EulerCharacteristicsForLinksOfSchubertCells} identifies $V_{\overline \sigma}$ and $A_{\sigma}$ and \cite[Corollary~5.4]{ShapiroVainshtein:1990:EulerCharacteristicsForLinksOfSchubertCells} gives the Euler characteristic of the latter space.
    Applying Poincar\'e duality identifies this with the Euler characteristic of compactly supported cohomology.

    One notes that the statement of \cite[Corollary~5.4]{ShapiroVainshtein:1990:EulerCharacteristicsForLinksOfSchubertCells} misses the special case $\sigma = w_0$.
    Indeed, one has to evaluate $P_\sigma(1)$.
By \cite[Theorem~5.3]{ShapiroVainshtein:1990:EulerCharacteristicsForLinksOfSchubertCells}, $P_\sigma(t) = \sum_{w \in W_{\sigma}} t^{n-n(w)} (1-t)^{n(w)-1} P_{\pi(\sigma,w)(t)}$, where $W_\sigma$, $n(w)$ and $\pi(\sigma,w)$ are defined in Sections~3.1 and~3.2 of \cite{ShapiroVainshtein:1990:EulerCharacteristicsForLinksOfSchubertCells}.
    If $\sigma^{-1}(n) \ne 1$, then $n(w)$ is always at least $2$ and hence $P_\sigma(1) = 0$.
    If $\sigma = w_0$, then one easily checks that $P_{w_0}(1) = 1$.
    Otherwise, one recursively will always eventually arrive at a $\pi(\sigma, w)$ of the first type, so that the polynomial has to vanish at $t = 1$.
\end{proof}

\begin{proof}[Proof of Theorem~\ref{thm:HH_Euler}]
    For a topological space $Y$ let $\chi_c(Y)$ be the Euler characteristic of the compactly supported cohomology of $Y$.
    We have to compute $\chi_c(\HochSpace)$.

    Let $\proj_1\colon \HochSpace \to X$ be the projection on the first factor and set $\HochSpace_w = \proj_1^{-1}(X_w)$.
    The Euler characteristic $\chi_c$ is additive, so that it suffices to compute $\chi_c(\HochSpace_w)$ for all $w \in W$.

    We can further decompose the spaces $\HochSpace_w$ by the relative positions of the flags $a$ and $ua$:
    we have a map $\phi$ from $\HochSpace_w$ to $X_w \times X_w$, sending $(a,b,u)$ to $(a, ua)$.
    For any $u \in U$, the fiber of $\phi$ over $(a, ua)$ is isomorphic to $C_{w_0}(a,ua) \times U_w$, where $U_w$ is the stabilizer of any point in $X_w$.
    
    Endowing $X_w \times X_w$ with the stratification by relative position, over each stratum $\phi$ is a fiber bundle.
    Away from the diagonal (i.e.~when $a \ne ua$), $\chi_c(C_{w_0}(a,ua) \times U_w) = 0$ by Lemma~\ref{lem:chi_of_intersection} above.
    Hence $\chi_c(\HochSpace_w \setminus \phi^{-1}(\Delta)) = 0$, where $\Delta$ is the diagonal in $X_w \times X_w$.
    On the other hand, for any $a \in X_w$ we have 
    \[
        \chi_c(\phi^{-1}(\Delta)) = \chi_c(C_{w_0}(a,a)) \times \chi_c(U_w) \times \chi_c(X_w) = 1.
    \]
    In total this give $\chi_c(\HochSpace) = \sum_{w \in W} \chi_c(\HochSpace_w) = \sum_{w \in W} 1 = |W| = n!$.
\end{proof}

\begin{Example}
    Let us conclude with a computation for $\SL 2$.
    Here $W = \symgroup 2 = \{1, w_0\}$, $X = \ps1$, $U = \Ga$, $X_1 = \{\infty\}$, and $X_{w_0} = \as1$.
    Further, $X_{w_0}^2 = X \times X \setminus \Delta$.

    Let us first describe the space $\HochSpace_1$, i.e.~$a=\infty$.
    As $X_1$ is a single point, we have $a = ua = \infty$ for any $u \in U$.
    Further $(0,b) \in X_{w_0}^2$ implies that $b \ne \infty$, i.e.~$b \in \ps1\setminus \{\infty\} \cong \as1$.
    In total we get
    \[
        \HochSpace_1 \cong \{\infty\} \times \as1 \times \Ga \cong \as2.
    \]
    The space $\HochSpace_{w_0}$ is more complicated.
    We have $a \in \as1 = \ps1 \setminus \{\infty\}$, $u \in \Ga$ arbitrary and $b \in \ps1$ with $b\ne a$ and $b \ne ua$.
    Thus for $u \ne 0$ we have $b$ in $\ps1$ minus $2$ points, while for $u=0$ we have $b \in \ps1 \setminus\{a\}$.

    Let us fix $a$ and look at the remaining options for $(b,u) \in (X\setminus\{a\}) \times \Ga \cong \as1 \times \as1$.
    Let $A$ be the space of all valid points $(b,u)$ and consider the projection $A \to \Ga$ to the $u$-coordinate.
    Its fiber over $0$ is $\as 1$, while the fiber over any other point is $\as1$ minus a point.
    Thus $A$ is homeomorphic to $\as1 \times \as1$ minus $\CC^\ast$.
    The compactly supported cohomology of $A$ is then easy to compute from the long exact sequence associated to this description.
    It is $\CC[-2] \oplus \CC[-3] \oplus \CC[-4]$.
    The compactly supported cohomology of $\HochSpace_{w_0}$ is then the same, but shifted up by a further $2$ (from the variation of $a \in X_{w_0}$).

    Putting the two computations together, we obtain a triangle
    \[
        \CC[-4] \oplus \CC[-5] \oplus \CC[-6] \to \pi_{\HochSpace,!}(k_{\HochSpace}) \to \CC[-4].
    \]
    Shifting down by $4l(w_0) = 4$, we arrive at
    \[
        \CC \oplus \CC[-1] \oplus \CC[-2] \to \HH^\bullet(\catDbConstr) \to \CC.
    \]
    As the zeroth Hochschild cohomology has to be isomorphic to $\bigoplus_n H^n(X,\CC) \cong k^2$ (see Remark~\ref{rem:cohomology_in_HH}), the connecting map in the long exact sequence of the triangle has to be trivial.
    Thus we arrive at an isomorphism of vector spaces
    \[
        \HH^n(\catDbConstr) = 
        \begin{cases}
            \CC & n = 2 \\
            \CC & n = 1 \\
            \CC^2 & n = 0 \\
            0 & \text{otherwise.}
        \end{cases}
    \]
\end{Example}

\sloppy
\printbibliography

\end{document}

%% file: macros.tex
\newcommand\sheaf\mathcal           
\newcommand\sO{{\sheaf{O}}}         
\newcommand\as[1]{\mathbb{A}^{#1}}  
\newcommand\ps[1]{\mathbb{P}^{#1}}  
\newcommand\pt{\mathrm{pt}}         
\newcommand\Ga{\mathbb{G}_a}        

\newcommand\cat\mathbf          
\newcommand\isoto{\xrightarrow{\sim}}   
\newcommand\Hom{\operatorname{Hom}}

\DeclareMathOperator\ev{ev}         
\DeclareMathOperator\coev{coev}     
\DeclareMathOperator\Tr{Tr}         
\DeclareMathOperator\PsId{Ps-Id}    
\DeclareMathOperator\id{id}         
\newcommand\FunctCont{\cat{Funct}_{\mathrm{cont}}}  
\newcommand\DGcatcont{\cat{DGCat}_{\mathrm{cont}}}  
\newcommand\dual\vee                
\newcommand\catVect{\cat{Vect}_k}   
\DeclareMathOperator\HH{HH}         

\newcommand\catO{\mathcal{O}}       
\newcommand\SL[1]{\mathrm{SL}_{#1}} 
\newcommand\symgroup[1]{{S_{#1}}}   

\newcommand\catDDMod[1]{\cat{DMod}(#1)}                         
\newcommand\catDbDModCoh[1]{\cat{DMod}_{\mathrm{coh}}(#1)}      

\newcommand\catDbDModCh{\cat{DMod}_{\{X_w\}}(X)}    
\newcommand\catDbStack{\catDbDModCoh{[X/U]}}        
\newcommand\catDbStackBig{\catDDMod{[X/U]}}         
\newcommand\catDbConstr{\cat{Sh}(X, \{X_w\})}       

\newcommand\HochSpace{{\mathfrak{H}}}               

\newcommand\proj{\mathrm{pr}}   
\newcommand\res[2]{\mathchoice{\left.#1\right|_{#2}}{#1|_{#2}}{#1|_{#2}}{#1|_{#2}}} 
\newcommand\rquot[2]{
    \mathchoice%
        {\left.#1\kern-0.2ex\middle/\kern-0.3ex\lower0.7ex\hbox{$\displaystyle #2$}\right.}%
        {\left.#1\middle/#2\right.}%
        {\left.#1\middle/#2\right.}%
        {\left.#1\middle/#2\right.}%
}
\newcommand\CC{{\mathbb{C}}}    
\newcommand\op{{\mathrm{op}}}   